\documentclass[12pt,twoside, final]{amsart}
\usepackage{amsmath,amsthm,amscd,amsfonts,amssymb,enumerate}
\usepackage{graphicx}
\usepackage{color}
\usepackage[colorlinks]{hyperref}

\usepackage[all,cmtip]{xy}
\usepackage{amsfonts}

\newtheorem{theorem}{Theorem}[section]
\newtheorem{proposition}[theorem]{Proposition}
\newtheorem{lemma}[theorem]{Lemma}
\newtheorem{corollary}[theorem]{Corollary}
\theoremstyle{definition}
\newtheorem{definition}[theorem]{Definition}

\theoremstyle{remark}
\newtheorem{remark}[theorem]{Remark}
\numberwithin{equation}{section} 

\DeclareMathOperator{\Sing}{Sing} 
\DeclareMathOperator{\Grass}{Grass}

 \DeclareMathOperator{\Pic}{Pic}
\DeclareMathOperator{\rk}{rk}
\DeclareMathOperator{\im}{im}

\begin{document}

\title[Bounds on the Dimension of the Brill-Noether ...]{Bounds on the Dimension of the Brill-Noether Schemes of Rank Two Bundles
 }

\author[Ali Bajravani]{Ali Bajravani}
\address[]{Department of Mathematics, Azarbaijan Shahid Madani University, Tabriz, I. R. Iran.\\
P. O. Box: 53751-71379}

\email{bajravani@azaruniv.ac.ir}

\maketitle
\vspace{-1cm}
\begin{abstract}
The aim of this note is to find upper bounds on the dimension of Brill-Noether locus' inside the moduli space of rank
two vector bundles on a smooth algebraic curve.
We deduce some consequences of these bounds.
\end{abstract}
\vspace{-0.3cm}
\section{Introduction}
Let $C$ be a projective smooth algebraic curve of genus $g$. For non-negative integers $n$ and $d$ we
denote by $U(n, d)$ the moduli space of stable vector bundles of rank $n$ and degree $d$, which is an irreducible scheme of dimension $n^2(g-1)+1$.
For an integer $k$ with $1\leq k\leq n+\frac{d}{2}$,
 the subset
$$B^k_{n, d}=\{E\in U(n, d)\mid h^0(E)\geq k  \}$$
of $U(n, d)$ inherits the structure of a closed sub-scheme of $U(n, d)$. With these notations, $B^k_{1, d}$ is the scheme of line bundles of degree $d$ with the space of global sections of dimension at least $k$, which is denoted commonly in literature by $W^{k-1}_{d}$.
 In the case of its non-emptiness, $B^k_{n, d}$ is expected to be of dimension
$n^2(g-1)+1-k(k-d+n(g-1))$.
As well, for a fixed line bundle $\mathcal{L}$ of degree $d$ we denote the sub-scheme of $U(n, d)$ parameterizing stable bundles $E\in B_{n, d}^k$ with
 determinant $\mathcal{L}$, by $B_{n, \mathcal{L}}^k$.

The schemes $B_{n, d}^k$, being as natural generalization of the Brill-Noether spaces of line bundles, as well as the spaces $ B_{n, \mathcal{L}}^k$,
have received wide attention from various authors.
However, in contrast with extensive results concerning these schemes, specifically the results on the non-emptiness and existence of components with minimum dimension, there are not, to our knowledge, systematic studies about
upper bounds for their dimensions, when $n\geq 2$.

We study this problem for Brill-Noether schemes of rank two bundles and we obtain upper bounds for $\dim B_{2, d}^k$ and $\dim B_{2, K}^k$, where $K$ denotes the canonical line bundle on $C$.

 The significant point in the rank two case is that a general element in a component of some $B_{2, d}^k$, which violates the upper bound and under some specified circumstances, might be assumed to be globally generated. Under the globally generated assumption, a result of Michael Atiyah is applicable. Based on the mentioned result, a globally generated vector bundle can be represented as an extension of a line bundle by the trivial line bundle.
Then, using the structure of tangent spaces of $B_{2, d}^k$, we relate the kernels of the Petri maps of appropriate bundles in suitable exact sequences. See Theorem \ref{bound theorem non-hyper complete}.
  As a byproduct, we obtain a Mumford type classification result. See Corollary \ref{Mumford thm for rank 2}.

As for the schemes $B_{2, K}^k$ we use an unpublished result of B. Feinberg, which might be considered as a refined version of Atiyah's result. See proposition \ref{feinberg proposition} and lemma \ref{lemma1}.

 By proving that for an arbitrary smooth curve $C$, a specific component
$X \subset B^2_{2, d}$ with prescribed circumstances,
 would be generically smooth of expected dimension;
our results push the results of Teixidor \cite{T. 1} and Flamini etal. \cite{C-F-K}, one step further.
See
remark \ref{remark 0}(c).

Similar problems, as the problems studied in this paper, have been studied for schemes of Secant Loci' in \cite{Bajravani 1}, \cite{Bajravani 2} and \cite{Bajravani 3} by the author.
\section{Preliminaries}
For $E\in B^k_{n, d}\setminus B^{k+1}_{n, d}$, the Petri map associated to $E$ controls the tangent vectors of $B^k_{n,d}$ at $E$. Indeed, the orthogonal of the image of the Petri map
\begin{align}\label{diag1}
\mu^2_{E}: H^0(E)\otimes H^0(K\otimes E^*)\rightarrow H^0(K\otimes E\otimes E^*),
\end{align}
 identifies the tangent space of $B^k_{n, d}$ at $E$. Similarly, the tangent space for $B^k_{2,K}$ is parameterized by the orthogonal of the image of the symmetric Petri map
 \begin{align}\label{diag2}
\mu^0_{s, E}: S^2H^0(E)\rightarrow H^0(S^2E).
\end{align}
See for example \cite{G}.

Assume that $E\in U(2, d)$ and
\begin{align}\label{diag0-3}
0\longrightarrow G \longrightarrow E \longrightarrow L \rightarrow 0,
\end{align}
is an exact sequence of bundles, with $G, L\in \Pic(C)$. Then, there exists a chain of bundles $S^2(E)\supset E^1\supset E^2\supset 0$, such that
\begin{align}\label{diag3-00}
\frac{S^2(E)}{E^1}=2L, \quad \frac{E^1}{E^2}=G\otimes L, \quad E^2=2G.
\end{align}
See \cite[Page 127]{H}. So, one has two exact sequences
\begin{align}\label{diag4}
0\longrightarrow 2G \longrightarrow S^2E \longrightarrow \frac{S^2E}{2G} \rightarrow 0,
\end{align}
\begin{align}\label{diag5}
0\longrightarrow G\otimes L \longrightarrow \frac{S^2E}{2G} \longrightarrow 2L \rightarrow 0.
\end{align}

An unpublished result of B. Feinberg, Lemma \ref{lemma1}, is the key tool in the proof of Theorem \ref{symmetric bound theorem non-hyper}. The lemma is a direct consequence of a characterization result, attributed to B. Feinberg.
We quote Teixidor's statement, \cite[Lemma 1.1]{T. 2}, of this characterizing result in Proposition \ref{feinberg proposition}.
 The proof we present for proposition \ref{feinberg proposition},
 is quoted from Feinberg's unpublished work in \cite{Feinberg}.
\begin{proposition}\label{feinberg proposition}
Denote by $F$ the greatest common divisor of the zeroes of the sections of $E$. Then, either there is a section of $E(-F)$
without zeroes or all sections of $E$ are sections of a line sub-bundle of
$E$.
\end{proposition}
\begin{proof}
The assertion is an immediate consequence of the following,

\textit{Claim I:} Assume that $s_1, \cdots s_{r+1}$ are base point free linearly independent sections of $E$ such that the space $\langle s_1, \cdots, s_{r+1}\rangle$ does not contain a nowhere vanishing section. Then,
there exists a line sub-bundle $L$ of $E$ such that $\langle s_1, \cdots, s_{r+1}\rangle$  is contained in $H^0(C, L)$.

\textit{Proof of Claim I:} Set $V:=\langle s_1, \cdots, s_{r+1}\rangle$ and consider the evaluation map
\begin{align}
e_V: C\times V \rightarrow E \quad , \quad e_V(p, s):=s(p).
\end{align}
We show that $\ker(e_V)$ is a vector bundle of rank $r$ and consequently the saturation of the image of $e_V$
is a line bundle. Observe that the hypothesis of being base point free is equivalent to the fact that the dimension of $\ker(e_V)_p$ is at most $r$ for all $p$ in $C$. If, on the other hand, the rank of $\ker(e_V)$
is generically less than $r$, then the dimension of the image of $\ker(e_V)$ under the composition:
$$\ker(e_V)\rightarrow C\times V \rightarrow V$$
is at most $r$. This, however, would imply that $V$ has a nowhere vanishing section, which is a contradiction.  Therefore, $\ker(e_V)$ is a vector bundle of rank $r$ and $e_V$ surjectively maps onto a line sub-bundle in $E$. This completes the proof of the \textit{Claim I}.
\end{proof}
\begin{lemma}\label{lemma1}
Any vector bundle $E$ with $h^0(E)=k\geq 2$ admits an extension as
\begin{align}\label{diag6}
0\longrightarrow \mathcal{O}(D) \stackrel{i}\longrightarrow E \stackrel{\pi}\longrightarrow L \rightarrow 0,
\end{align}
where $D$ is an effective divisor and either $h^0(\mathcal{O}(D))=1$ or $h^0(\mathcal{O}(D))=k$.
\end{lemma}

Motivated by Lemma \ref{lemma1}, two types of bundles with sections are distinguishable.
\begin{definition}\label{definition 1}
A vector bundle $E$ with $h^0(E)=k\geq 1$ will be said of first type if it admits an extension as (\ref{diag6}) with $h^0(\mathcal{O}(D))=k$. Otherwise we call $E$ of second type.
\end{definition}
\section{Main results}\label{section3}
\begin{theorem}\label{bound theorem non-hyper complete}
 Let $k$, $d$ be integers with $3\leq d\leq 2g-2$, $2\leq k\leq 2+\frac{d}{2}$.
Then,
  \begin{align}\label{diag7-1}
\dim B^k_{2, d}\leq 2(g-1)+d-2k+1.
\end{align}
\end{theorem}
\begin{proof}
Observe first that if a general element $E$ of an irreducible component $X$ of $B^k_{2, d}$ satisfies $h^0(E)\geq k+1$, then we can consider $X$ as a component of $B^{k+1}_{2, d}$. Therefore for general $E\in X$ one may assume $h^0(E)=k$.
Assume that $d$ is a minimum integer such that for some suitable $k$ there exists a
component $X$  of $B^k_{2, d}$ with $\dim X\geq 2(g-1)+d-2k+2$.
 Then, a general element $E$ in $X$ is globally generated. Indeed otherwise we obtain
 $ \dim B^{k-1}_{2, d-2}\geq 2(g-1)+d-2k+2$,
 which is impossible by minimality of $d$.
Therefore, by \cite[Theorem 2]{A}, a general element $E$ in $X$ has a trivial line bundle as its line sub-bundle. Furthermore $E$ admits a representation as
 \begin{align}\label{diag3-0-0}
0\longrightarrow \mathcal{O}_C \stackrel{i}\longrightarrow  E \stackrel{\pi}\longrightarrow L \longrightarrow 0,
\end{align}
with the property that the sections of $L$ belonging to the image of $H^0(\pi)$
 have at most one number of base points. Indeed, if $L$ has the points $p, q$ as its base points, then $h^0(E(-p-q))\geq k-2$. This implies that
$$\dim B^{k-2}_{2, d-4}\geq 2(g-1)+d-2k+2,$$
 which is absurd again by minimality of $d$.
Take an extension as (\ref{diag3-0-0}) and consider the exact sequence
\begin{align}\label{diag3}
0\longrightarrow H^0(\mathcal{O}_C) \stackrel{H^0(i)}\longrightarrow H^0(E) \stackrel{H^0(\pi)}\longrightarrow V \longrightarrow 0,
\end{align}
 where $V$ is the image of the map $H^0(\pi): H^0(E) \longrightarrow H^0(L)$.
  The exact sequence (\ref{diag3}) together with various Petri maps gives rise to a commutative diagram as
\begin{align}\label{diag -11}
\xymatrix { H^0(\mathcal{O}_C)\otimes H^0(K\otimes E^*)\stackrel{f_{1}}\longrightarrow\ar[d]^{\mu}&
\!\!\!\!\!\!\!\!\!\!\!\!\!\!\!\!\!\!\!\!\!\!\!\!\!\!\!\!\!\!\!\!\!\!\!\!\!\!\!\!\!\!\!\!\!
\!\!\!\!\!\!\!\!\!\!\!\!\!\!\!\!\!\!\!\!\!\!\!\!\!\!\!\!\!\!
\!\!\!\!\!\!\!\!\!\!\!\!\!\!\!\!\!\!\!\!\!\!\!\!\!\!\!\!\!\!\!
H^0(E)\otimes H^0(K\otimes E^*)\stackrel{g_{1}}\longrightarrow \!\!\!\!\!\!\!\!\!\!\!\!\!\!\!\!\!\!\!\!\!\!\!\!\!\!\!\!\!\!\!\!\!\!\!\!\!\!\!\!\!\!\!\!\!
\!\!\!\!\!\!\!\!\!\!\!\!\!\!\!\!\!\!\!\!\!\!\!\!\!\!\!\!\!\!\!\!\!\!\!\!\!\!\!\!
\!\!\!\!\!\!\!\!\!\!\!\!\!\!\!\!\!\!\!\!\!\!
\ar[d]^{\mu^2_{E}}&V\otimes H^0(K\otimes E^*)\ar[d]_{\mu_{L, V}}\\
 H^0(\mathcal{O}_C\otimes K\otimes E^*)\ar[r]^{f_{2}}&
H^0(K\otimes E\otimes E^*)\quad\stackrel{g_{2}}\longrightarrow &\!\!\!\!\!\!\!\!\!\!\! H^0(K\otimes L\otimes E^*),&\!\!\!\!\!\!\!\!\!\!\!\!\!\!\!\!\!\!\!\!\!\!\!\!\!\!\!\!\\
}
\end{align}
in which the maps $f_{1}$ and $f_{2}$ are injective and $g_{1}$ is surjective. Observe furthermore that the map $\mu$ is an isomorphism. The snake lemma applied to this situation implies that
\begin{align}\label{diag3-1}
\dim \ker \mu^2_{E}=\dim \ker \mu_{L, V}.
\end{align}
According to the assumption concerning dimension of $X$, we obtain
$$\dim \ker \mu_{L, V}\geq (k-1)(2g-2-d+k-1).$$
Assuming $V=\langle v_1, \cdots, v_{k-1} \rangle$ and setting
$$V_i:=\langle v_1, \cdots, v_i \rangle \quad , \quad i=2, \cdots, k-1,$$
we would have
$\dim \ker \mu_{L, V_i}-\dim \ker\mu_{L, V_{i-1}}\leq h^0(K\otimes E^*).$
These together with the base point free pencil trick applied to the map
$$\mu_{L, V_2}:V_2\otimes H^0(K\otimes E^*)\rightarrow H^0(K\otimes L\otimes E^*),$$
implies $h^0(K\otimes E^*\otimes L^*(B))\geq 2(2g-2-d+k)-(k-1)$, where $B$ is the base locus of the sections of $V_2$. Note also that $0\leq \deg(B)\leq 1$. Therefore,
\begin{align}\label{diag3-2}
h^0(K\otimes E^*\otimes L^*)\geq 2(2g-2-d+k)-k.
\end{align}

If  $\frac{4g-3}{3}\leq d$ then, as $\deg(K\otimes E^*\otimes L^*)< 0$ and $K\otimes E^*\otimes L^*$ is stable, one has $h^0(K\otimes E^*\otimes L^*)=0$, which is in contradiction with inequality (\ref{diag3-2}).\\

  Recall that
$h^0(E\otimes L)=h^0(K\otimes E^*\otimes L^*)+3d-2(g-1)\geq 2(g-1)+d+k.$
Now if $d\leq \frac{4g-4}{3}$, then $\mu(E \otimes L)\leq 2g-3$. Observe furthermore that $E \otimes L$ is stable.
As a consequence of Propositions 3 and 4 of \cite{Re}, the
 Clifford theorem for vector bundles for such a this situation asserts that $h^0(E\otimes L)\leq \frac{\deg(E\otimes L)+\rk(E\otimes L)}{2}$,
  by which we obtain
$2(g-1)+d+k\leq 1+\frac{3d}{2}.$
Consequently we get $d+k\leq 0$, which is absurd.
\end{proof}
\begin{theorem}\label{dimension equality}
If 
$g\geq 5$, then
\begin{align}\label{diag3-3-1-1}
\dim B^2_{n, d}\leq n(n-1)(g-1)+d-3.
\end{align}
\end{theorem}
\begin{proof}
 Assume that $X$ is an irreducible component of $B^2_{n, d}$ and $E$ is a general element of $X$. Assume moreover, as in theorem \ref{bound theorem non-hyper complete}, that a general element $E\in X$ satisfies $h^0(E)=2$.
 Observe that, using a diagram as in diagram (\ref{diag -11}),
we can obtain an equality as (\ref{diag3-1}), by which,
 if $E$ turns out to be of second type, then
 $\mu^2_E$ would be injective. So $X$ has to be generically smooth and it has to have the expected dimension, which is certainly smaller than the claimed bound.

If a general element of $X$ turns to be of first type, then
$$\dim X\leq n(n-1)(g-1)+d-4.$$
Indeed, if a general element $E\in X$ admits a presentation as
$$0\rightarrow H \rightarrow E \rightarrow F \rightarrow 0,$$
where $H$ is a line bundle with $h^0(H)=2, \deg(H)=d_1$ and $\rk(F)=n-1$, then since the stable bundles deform to non-stable ones, we can assume in counting that $F$ is stable as well. So the dimension of the set of bundles as $F$, is bounded by $\dim U(n-1, d-d_1)=(n-1)^2(g-1)+1$.
Meanwhile, the line bundles as $H$ would vary in a subset $\mathcal{H}$ of $B^2_{1, d_1}$ and the Martens' theorem asserts that $\dim \mathcal{H}\leq d_1-2$ ($\dim \mathcal{H}$ can be $d_1-2$ if $C$ is hyper-elliptic and $\dim \mathcal{H}\leq d_1-3$ otherwise). Therefore the dimension of $X$ would be bounded
by
$$[d_1-2]+[(n-1)^2(g-1)+1]+(h^1(H\otimes F^*)-1).$$
Observe that $h^1(H\otimes F^*)=(n-1)(g-1)+d-nd_1$ by Riemann-Roch. Moreover $d_1\geq 2$ and so
$$\dim X\leq
n(n-1)(g-1)+d-2-d_1(n-1)\leq n(n-1)(g-1)+d-4,$$
 as required.
\end{proof}
%
Motivated by \cite[Theorem 1.2]{C-F-K}, one can sharpen the bound in Theorem \ref{bound theorem non-hyper complete} under some restrictions on the numbers $r, d$, as
\begin{theorem}\label{bound theorem non-hyper complete 2}
 Let $k$, $d$ be integers with $3\leq d\leq 2g-2-\frac{k}{2}$, $2\leq k\leq 2+\frac{d}{2}$.
Then,\\
 if $k\geq 3$, then
 $\dim B^k_{2, d}\leq 2g+d-4k.$
 While for $k=2$, the integer $d$ can vary in the set $\{3, \cdots, 2g-5  \}$ with the same bound for $\dim B^2_{2, d}$.
\end{theorem}
\begin{proof}
The argument of proof of theorem \ref{bound theorem non-hyper complete} goes through to deduce the result. Notice that the further restriction
on $d$ in the case $k=2$ was needed to be imposed, because the quantity $2g+d-4k$ turns out to be smaller than the expected dimension for $2g-4\leq d\leq 2g-2$.
\end{proof}
\subsection{The case of canonical determinant}
\begin{theorem}\label{symmetric bound theorem non-hyper}
For an integer $k$ with $2\leq k\leq g+1$, any irreducible component $X$ of $B_{2, K}^k$ satisfies
\begin{align}\label{bound inequality, symmetric}
\dim X\leq 3g-2k-2.
\end{align}
\end{theorem}
\begin{proof}
Let $X$ be an irreducible component of $B_{2, K}^k$ and a general element $E$ of $X$ satisfies $h^0(E)=k$.
 Assume that a general member $E\in X$
is of second type
  and set $\gamma:=\dim X$. Then, one has
\begin{align}
3g-3-(\dim \im \mu^0_{s, E})\geq \gamma,
\end{align}
  where $\mu^0_{s, E}$ is the symmetric Petri map associated to $E$ as in (\ref{diag2}). So
\begin{align}\label{diag31}
\dim \ker \mu^0_{s, E}\geq \gamma +\frac{k(k+1)}{2}+3-3g.
\end{align}
The exact sequence
$0\rightarrow \mathcal{O}(2D)\rightarrow S^2E\rightarrow \frac{S^2E}{\mathcal{O}(2D)} \rightarrow 0,$
arising from the exact sequence (\ref{diag4}),
 gives rise to a commutative diagram as
\begin{align}\label{diag9}
\xymatrix {
0 \rightarrow S^2H^0(\mathcal{O}(D))\ar[r]\ar[d]^{\mu^0_{s, D}}&S^2H^0(E)\ar[d]^{\mu^0_{s, E}} \ar[r]&\frac{S^2H^0(E)}{S^2H^0(\mathcal{O}(D))} \ar[d]^{\mu}\rightarrow 0\\
0 \rightarrow H^0(\mathcal{O}(2D))\ar[r]&H^0(S^2E)\ar[r]&H^0(\frac{S^2E}{\mathcal{O}(2D)})\rightarrow ...
}
\end{align}
Since $S^2H^0(\mathcal{O}(D))=\mathbb{C}$, the map
 $\mu^0_{s, D}$
turns to be injective. This together with the snake lemma gives an inequality as
\begin{align}\label{diag9-1}
\dim \ker \mu \geq \dim \ker \mu^0_{s, E}.
\end{align}
Therefore, using the inequality (\ref{diag31}) we obtain
\begin{align}\label{diag10}
\dim \ker \mu \geq \gamma +\frac{k(k+1)}{2}+3-3g.
\end{align}
Let $V$ be as in the proof of Theorem \ref{bound theorem non-hyper complete} and observe by effectiveness of $D$ that the vector space $V$ can be considered as a subspace of
$H^0(\mathcal{O}(D)\otimes L)$. Similar to the previous argument, the exact sequence
$$ 0\rightarrow \mathcal{O}(D)\otimes L\rightarrow \frac{S^2E}{\mathcal{O}(2D)}\rightarrow 2L \rightarrow 0,$$
as well arising from the exact sequence (\ref{diag5}),
together with the equality
$$S^2H^0(E)
=S^2V\oplus V\oplus \mathbb{C},$$
leads to the following commutative diagram of bundles
\begin{align}\label{diag11}
\xymatrix {
0 \rightarrow V\ar[r]\quad \ar[d]^{\theta}&\frac{S^2H^0(E)}{S^2H^0(\mathcal{O}(D))}\ar[d]^{\mu} \ar[r]&\quad S^2V \ar[d]^{\mu^0_{s, V, L}} \rightarrow 0\\
0 \rightarrow H^0(\mathcal{O}(D)\otimes L)\ar[r]&H^0(\frac{S^2E}{\mathcal{O}(2D)})\ar[r]&\quad H^0(2L)\rightarrow \cdots,}
\end{align}
where $\mu^0_{s, V, L}$ is the symmetric Petri map of $L$ restricted to $S^2V$ and $\theta$ is the inclusion map.
Once again, as a consequence of the injectivity of $\theta$ and the snake lemma, we obtain
\begin{align}\label{diag9-2}
\dim \ker \mu^0_{s, V, L} \geq \dim \ker \mu,
\end{align}
by which together with (\ref{diag10}) an inequality as
\begin{align}\label{diag12}
\dim \ker \mu^0_{s, V, L} \geq \gamma +\frac{k(k+1)}{2}+3-3g
\end{align}
would be obtained.
This, in combination with $\dim \ker \mu^0_{s, V, L} \leq \dim S^2V-\dim V$, implies
\begin{align}\label{diag12-1}
\gamma \leq 3g-2k-2,
\end{align}
as required.

Finally if $\dim X\geq 3g-2k-1$ then
 a general member $E$ of $X$ fails to be of first type. Indeed otherwise, assume that a general member $E\in X$ admits a presentation as
\begin{align}
0\longrightarrow \mathcal{O}(D) \stackrel{i}\longrightarrow E \stackrel{\pi}\longrightarrow K\otimes \mathcal{O}(-D) \rightarrow 0,
\end{align}
 with $\deg(D)=t$. Then, the stability of $E$ implies that $t\leq g-2$ and we would have
 $$\dim B^{k}_{1,t}+h^1(\mathcal{O}(D)\otimes L^{-1})-1\geq 3g-2k-1.$$
This, since $h^0(\mathcal{O}(D)\otimes L^{-1})=0$ by stability of $E$, implies that
$$\dim B^{k}_{1, t}\geq 2t-2k+3,$$
which is absurd by Martens' theorem.
\end{proof}
\section{Remarks and Corollaries}
\begin{corollary}\label{Mumford thm for rank 2}(\textbf{Mumford's Theorem for rank two bundles})
If $C$ is non-hyper elliptic of genus $g\geq 19$ and if for some $k, d$ with
$0< 2k-2\leq d\leq 2g-\frac{3}{2}k-\frac{7}{2}$ one had $\dim B_{2, d}^k=2g+d-4k$, then either $C$ is trigonal, or bi-elliptic, or a smooth plane quintic.
\end{corollary}
\begin{proof}
Assume that $X$ is an irreducible component of $B_{2, d}^k$ with $\dim X=2g+d-4k$.
If a general
element $E\in X$ is of first type and has $k$ number of independent sections, then one has $\dim B_{1, t}^{k}\geq 2(t-2k+1)$ for some
integer $t$ with $0<2(k-1)\leq t\leq g-2$. This, by Mumford's theorem, might occur only if $t-2k+1=0$ by which the equality $\dim B_{1, 2k-1}^{k}=0$ holds. So $\mathcal{O}(D)\in B_{1, 2k-1}^{k}$, which may happen only in the case that
either $C$ is trigonal, or bi-elliptic, or a smooth plane quintic.

\textit{Claim II:}
If $E$ fails to be of first type, then for general points $p_1, \cdots p_{[\frac{k-2}{2}]}\in C$, the stable vector bundle $E(-p_1-\cdots -p_{[\frac{k-2}{2}]})$
would fail to admit an extension of first type.

\textit{Proof of Claim II:}  Assume first that $k$ is even. If
the stable vector bundle $E(-p_1-\cdots -p_{[\frac{k-2}{2}]})$
turns to be of first type, then
there exists a set of line bundles $H$ with $h^0(H)\geq 2$ and $\deg{H}\leq g-2$. Tensoring $H$ with $\mathcal{O}(p_1+\cdots +p_t)$ for general points $p_1+\cdots +p_t$, if necessary, we can assume that $H\in B^2_{1, g-2}$. Therefore we obtain $\dim B^2_{1, g-2}\geq 2g+d-4k+[\frac{k-2}{2}]$. This by Martens' theorem implies that $7k-8\geq 2g+2d$. On the other hand, the inequalities $2k-2\leq d$ and $2k-2\leq 2g-\frac{3}{2}k-\frac{7}{2}$ imply $4k\leq 2d+4$ and $3k\leq \frac{12}{7}g-\frac{9}{7}$, respectively. Summing up all the inequalities we obtain $g\leq 18$, which is absurd. If $k$ is an odd number, then the argument goes verbatim to prove the claim by replacing $B^3_{1, g-2}$ with
$B^2_{1, g-2}$. So the \textit{Claim II} is established.

If a general bundle $E\in X$ turns to be of second type and if $k=2n$, then the scheme
 $B_{2, d-2[\frac{k-2}{2}]}^{2}$ contains a subset $Y$ which is at least of dimension $2g+d-4k+[\frac{k-2}{2}]$ and its general member is a vector bundle of second type.
According to the work of M. Teixidor in \cite{T. 1} such a subset $Y$, if non-empty, is of expected dimension and the expected dimension is strictly smaller than
$2g+d-4k+[\frac{k-1}{2}]$ for $d$ in the given range. This is a contradiction.

 If $k=2n+1$, with similar assumption on $E$ the scheme
 $B_{2, d-2[\frac{k-2}{2}]}^{3}$ would contain a subset $Y$
which is at least of dimension $2g+d-4k+[\frac{k-2}{2}]$ and its general member is a vector bundle of second type.
 This possibility can be excluded by another work of M. Teixidor in \cite{T. 2}.
\end{proof}
\begin{corollary}
The scheme $B^2_{2, K}$ is reduced and irreducible of dimension $3g-6$.
\end{corollary}
\begin{proof}
The upper bound $3g-6$ on the dimension is obvious by theorem \ref{symmetric bound theorem non-hyper}.
If $E\in B^2_{2, K}\setminus B^3_{2, K}$, then the petri map $\mu^0_{2, K, E}$ turns to be injective. Indeed, if $E$ is a bundle of first type, then using diagram (\ref{diag9}),
 since $\frac{S^2H^0(E)}{S^2H^0(\mathcal{O}(D))}$ vanishes,
the Petri map $\mu^0_{s, E}$ would be injective. While if $E$ is of second type, since $S^2V$ is one dimensional, then $\mu^0_{s, V, L}$ is injective and so the map $\mu$ is injective by (\ref{diag9-2}). This together with (\ref{diag9-1}) implies that
the Petri map $\mu^0_{s, E}$ is again injective. So we obtain
\begin{align}\label{50}
\Sing B^2_{2, K}=B^3_{2, K}.
\end{align}
Since $B^2_{2, K}$ is of expected dimension, so it might be reducible only if its singular locus is, by \cite{T.3}, of codimension $\leq 1$; i.e. $\dim B^3_{2, K}\geq 3g-7$, by (\ref{50}).
This is a contradiction, because by Theorem (\ref{symmetric bound theorem non-hyper}) the locus $ B^3_{2, K}$ is of dimension at most $3g-8$.

Since, again by theorem \ref{symmetric bound theorem non-hyper}, no irreducible component of $B^2_{2, K}$ is contained entirely in $B^3_{2, K}$, so $B^2_{2, K}$ would be reduced.
\end{proof}

Using Lemma \ref{Mumford Thm for Fixed determinantals}, the bound in theorem \ref{symmetric bound theorem non-hyper} can be sharpened for odd values of $k$.
\begin{lemma}\label{Mumford Thm for Fixed determinantals}
 If $\mathcal{L}$ is a globally generated line bundle on $C$ with $h^0(\mathcal{L})=s+1$, then the set of vector bundles of second type $E\in B^k_{2, d, \mathcal{L}}$ ($k=1, 2$),
 if non-empty, is of dimension at most
 $s+d-4 $, (res. at most of dimension $\frac{d}{2}+2(s-3)$), if $k=2$ (res. if $k=3$).
\end{lemma}
\begin{proof}
For $k=2$, with notations as in proof of \cite[Page 124]{T. 1}, the dimension of the set of vector bundles $E\in B_{2, d, \mathcal{L}}^k$ of second type
is bounded by
  $$
\dim \lbrace D \rbrace  +   h^0(\mathcal{L}(-D))-1 + \dim \langle \acute{D} \rangle - (h^0(E)-1)
$$
where $\acute{D}$ is a divisor in the linear series $\vert \mathcal{L}(-D) \vert$ and $t=\deg(D)$.
It is now an easy argument to see that this quantity is bounded by
$$t + (s-1) +  (d-t-2)-1=d+s-4.$$
  If $k=3$, then a close analysis in the proof of \cite[Theorem 2]{T. 2}, implies that the dimension of the bundles $E\in \dim B^3_{2, d, \mathcal{L}}$ which are of second type, is bounded by the quantity
$\dim \lbrace D \rbrace + \dim \Grass(2, \mathcal{L}(-D)) +\dim \langle \acute{D_1}\cap \acute{D_2}\rangle - (h^0(E)-1)
\leq \frac{d}{2}+2(s-3),$
as required.
\end{proof}
\begin{corollary}
If $k$ is odd, then
\begin{align}\label{diag11}
\dim B^k_{2, K}\leq 3g-2k-3.
\end{align}
\end{corollary}
\begin{proof}
An irreducible component $X$  of $B_{2, K}^k$ whose general member is a bundle of first type
has dimension $\leq 3g-2k-3$, because otherwise one obtains $\dim B_{1, t}^{k}\geq 2t-2k+2$ for some $k$ and $t$ with
$0<2k-2\leq t\leq g-2$. This is obviously absurd.

Assume that $\dim X=3g-2k-2$ and set $k-1=2n$. \\
\textit{Claim III:} If a general $E\in X$ fails to be of first type, then for general points $p_1, \cdots p_{i}\in C$ with $1\leq i\leq \frac{k-1}{2}$ the stable vector bundle $E(-p_1-\cdots -p_i)$
would fail to admit an extension of first type. \\
The proof of \textit{Claim III} is similar to the proof of \textit{Claim II} in corollary \ref{Mumford thm for rank 2}.

 Lemma (\ref{Mumford Thm for Fixed determinantals}) together with \textit{Claim III}
implies that if a general element of $X$ fails to be of first type
 then
$$3g-2k-2\leq (n-1)+\dim B^3_{2, 2g-2n, K(-2p_1-\cdots-2p_{n-1})}\leq 3g-4n-5,$$
 which is absurd.
\end{proof}
\begin{remark}\label{remark 0}
(a) If $C$ is an arbitrary $3$-gonal curve, then Theorem \ref{bound theorem non-hyper complete 2} together with Theorem \cite[Thm. 1.2(b)]{C-F-K}
imply $\dim B^2_{2, d}=2g+d-8$. Indeed, Theorem \cite[Thm. 1.2(b)]{C-F-K} establishes this result for a general $3$-gonal curve and
so for non-general $3$-gonal curves, one has $\dim B^2_{2, d}\geq 2g+d-8$. Now, Theorem \ref{bound theorem non-hyper complete 2} applied to such a non-generic curve implies
the equality for any $3$-gonal curve.

\noindent (b) According to theorem \ref{dimension equality}, one immediately re-obtains
$\dim B^2_{n, n(g-1)}=n^2(g-1)-3$. Meanwhile, by the same theorem, an immediate prediction suggests the quantity $n(n-1)(g-1)+d-2k+1$ as a bound to the dimension of $B^k_{n, d}$ when $n\geq 3, k\geq 3$. A proof to this expectation is unknown to me. Such a bound re-obtains Marten's bound on the dimension of the Brill-Noether schemes of line bundles.

\noindent (c) The proofs of theorems \ref{bound theorem non-hyper complete} and \ref{dimension equality} indicate that the Petri map is injective at the bundles $E\in B^2_{n, d}$ which are of second type. Therefore
$$\Sing B^2_{n, d}\subseteq B^3_{n, d}\cup \mathcal{E}_1, $$
 where $\mathcal{E}_1$ denotes the set of bundles $E\in B^2_{n, d}$ of first type.
 This reproves the generic smoothness of the locus' introduced by Teixidor in \cite{T. 1} and Flamini etal. in \cite{C-F-K}.
\end{remark}
\textbf{Acknowledgment:} The author wishes to thank F. Flamini, P. Newstead and M.Teixidor for their valuable hints and for sharing their knowledge. I specially thank G. H. Hitching whose careful reading and comments changed the previous manuscript of this paper, considerably.
 Teixidor supported me by sending a draft of the unpublished paper \cite{Feinberg} at the right time; to her, I express my double gratitude. 

\end{document}